\documentclass[11pt,reqno,tbtags,a4paper]{amsart}
\usepackage{amssymb}
\usepackage{url}
\usepackage[square,numbers]{natbib}
\bibpunct[, ]{[}{]}{;}{n}{,}{,}

\title
{Mean and variance of balanced P\'olya urns}

\date{19 February, 2016; revised 17 December, 2019}

\author{Svante Janson}
\thanks{Partly supported by the Knut and Alice Wallenberg Foundation}
\address{Department of Mathematics, Uppsala University, PO Box 480,
SE-751~06 Uppsala, Sweden}
\email{svante.janson@math.uu.se}
\urladdr{http://www.math.uu.se/svante-janson}


\numberwithin{equation}{section}

\renewcommand\le{\leqslant}
\renewcommand\ge{\geqslant}

\allowdisplaybreaks

\theoremstyle{plain}
\newtheorem{theorem}{Theorem}[section]
\newtheorem{lemma}[theorem]{Lemma}

\theoremstyle{definition}
\newtheorem{example}[theorem]{Example}

\newtheorem{problem}[theorem]{Problem}
\newtheorem{remark}[theorem]{Remark}

\newtheorem*{ack}{Acknowledgement}

\theoremstyle{remark}

\newenvironment{romenumerate}[1][-10pt]{
\addtolength{\leftmargini}{#1}\begin{enumerate}
 }{\end{enumerate}}

\newcounter{oldenumi}
{\setcounter{oldenumi}{\value{enumi}}
\begin{romenumerate} \setcounter{enumi}{\value{oldenumi}}}
{\end{romenumerate}}

\newcounter{thmenumerate}
\newenvironment{thmenumerate}
{\setcounter{thmenumerate}{0}%
 \def\item{\par
 \refstepcounter{thmenumerate}\textup{(\roman{thmenumerate})\enspace}}
}
{}

\newcounter{xenumerate}   

\newcommand\pfitemx[1]{\par#1:}
\newcommand\pfitemref[1]{\pfitemx{\ref{#1}}}

\newcommand{\refT}[1]{Theorem~\ref{#1}}
\newcommand{\refTs}[1]{Theorems~\ref{#1}}

\newcommand{\refL}[1]{Lemma~\ref{#1}}

\newcommand{\refR}[1]{Remark~\ref{#1}}
\newcommand{\refS}[1]{Section~\ref{#1}}
\newcommand{\refSs}[1]{Sections~\ref{#1}}

\newcommand{\refE}[1]{Example~\ref{#1}}

\newcommand{\refApp}[1]{Appendix~\ref{#1}}

\begingroup
  \divide\count255 by 60
  \multiply\count255 by -60
  \advance\count255 by \time
  \ifnum \count255 < 10 \xdef\klockan{\the\count1.0\the\count255}
  \else\xdef\klockan{\the\count1.\the\count255}\fi
\endgroup

\newcommand{\sumko}{\sum_{k=0}^\infty}

\newcommand{\sumk}{\sum_{k=1}^\infty}

\newcommand{\sumin}{\sum_{i=1}^n}

\newcommand\set[1]{\ensuremath{\{#1\}}}

\newcommand\xpar[1]{(#1)}
\newcommand\bigpar[1]{\bigl(#1\bigr)}
\newcommand\Bigpar[1]{\Bigl(#1\Bigr)}

\newcommand\lrpar[1]{\left(#1\right)}

\def\rompar(#1){\textup(#1\textup)}    

\newcommand\parfrac[2]{\lrpar{\frac{#1}{#2}}}

\newcommand\Bigparfrac[2]{\Bigpar{\frac{#1}{#2}}}

\def\xexp(#1){e^{#1}}
\newcommand\ceil[1]{\lceil#1\rceil}

\newcommand\ntoo{\ensuremath{{n\to\infty}}}

\newcommand\jtoo{\ensuremath{{j\to\infty}}}
\newcommand\too{\to\infty}

\newcommand\norm[1]{\|#1\|}

\newcommand\punkt{.\spacefactor=1000}    
    
\newcommand\ie{i.e\punkt}
\newcommand\eg{e.g\punkt}

\newcommand{\as}{a.s\punkt}

\newcommand\ii{\mathrm{i}}

\newcommand{\tend}{\longrightarrow}
\newcommand\dto{\overset{\mathrm{d}}{\tend}}

\newcommand\asto{\overset{\mathrm{a.s.}}{\tend}}

\newcommand\bbR{\mathbb R}
\newcommand\bbC{\mathbb C}

\newcounter{CC}
\newcounter{cc}

\renewcommand\Re{\operatorname{Re}}
\renewcommand\Im{\operatorname{Im}}

\newcommand\E{\operatorname{\mathbb E{}}}
\renewcommand\P{\operatorname{\mathbb P{}}}

\newcommand\Var{\operatorname{Var}}

\newcommand\ga{\alpha}

\newcommand\gd{\delta}
\newcommand\gD{\Delta}

\newcommand\gG{\Gamma}

\newcommand\gl{\lambda}

\newcommand\gs{\sigma}
\newcommand\gS{\Sigma}

\newcommand\eps{\varepsilon}

\renewcommand\phi{\xxx}  

\newcommand\cF{\mathcal F}

\newcommand\smatrixx[1]{\left(\begin{smallmatrix}#1\end{smallmatrix}\right)}

\newcommand\qw{^{-1}}
\newcommand\qww{^{-2}}
\newcommand\qq{^{1/2}}
\newcommand\qqw{^{-1/2}}

\newcommand\intoi{\int_0^1}
\newcommand\intoo{\int_0^\infty}

\newcommand\ooi{(0,1]}
\newcommand\ooo{[0,\infty)}

\newcommand\dd{\,\mathrm{d}}
\newcommand\ddx{\mathrm{d}}

\newcommand\lhs{left-hand side}
\newcommand\rhs{right-hand side}

\newcommand\etto{\bigpar{1+o(1)}}

\newcommand\sumgl{\sum_{\gl}}
\newcommand\sumgla{\sum_{\gl\in\gs(A)}}
\newcommand\sumglx{\sum_{\gl\neq\gl_1}}
\newcommand\summu{\sum_{\mu}}
\newcommand\summux{\sum_{\mu\neq\gl_1}}
\newcommand\sumiq{\sum_{i=1}^q}
\newcommand\sumjq{\sum_{j=1}^q}
\newcommand\tin{T_{i,n,\gl,\mu}}
\newcommand\txn{T_{\ceil{xn},n,\gl,\mu}}
\newcommand\tyn{T_{\ceil{n^y},n,\gl,\mu}}
\newcommand\tqn[1]{T_{#1,n,\gl,\mu}}
\newcommand\tynb{T_{\ceil{n^y},n,\gl,\bgl}}
\newcommand\gsa{\gs(A)}
\newcommand\gsax{\gs(A)\setminus\set{\gl_1}}
\newcommand\xnn{_{\ceil{xn},n}}

\newcommand\PIx{\widehat P}
\newcommand\nyn{_{\ceil{n^y},n}}
\newcommand\bgl{\overline \gl}
\newcommand\mm{^{(m)}}
\newcommand\wmm{\frac{1}{m!}}
\newcommand\uniz{uniformly for $z$ in any fixed compact set in the complex plane}
\newcommand\qll{^{(\ell)}}
\newcommand\el{E_\lambda}
\newcommand\nl{N_\lambda}
\newcommand\pl{P_\lambda}
\newcommand\pli{P_{\lambda_1}}
\newcommand\nul{\nu_\gl}
\newcommand\tA{\widetilde A}

\newcommand\hF{\widehat F}
\newcommand\hA{\widehat A}

\newcommand{\Polya}{P\'olya}

\begin{document}

\begin{abstract} 
It is well-known that in a small P\'olya urn, i.e., an urn where second
largest real part of an eigenvalue is at most half the largest eigenvalue, 
the distribution of the numbers of balls of different colours in the urn is
asymptotically normal under weak additional conditions. 
We consider the balanced case, and then 
give asymptotics of the mean and the covariance matrix, 
showing that after appropriate normalization, 
the mean and covariance matrix converge to the mean and variance of the
limiting normal distribution. 
\end{abstract}

\maketitle

\section{Introduction}\label{S:intro}

A (generalized) \Polya{} urn contains balls of different colours.
A ball is drawn at random from the urn, and is replaced by a set of balls
that depends on the colour of the drawn balls. (Moreover, the replacement
set may be random, with a distribution depending on the drawn colour). 
This is repeated an infinite number of times, and we are interested in the
asymptotic composition of the urn.
For details, and the assumptions used in the present paper, see
\refS{Spolya}; for the history of \Polya{} urns, see \eg{} \citet{Mahmoud}.
 
It is well-known, and proved 
under various conditions
in a number of papers by  a variety of authors,
see \eg{} \cite[Theorems 3.22--3.24]{SJ154}, 
that the asymptotic behaviour depends on the eigenvalues of
the \emph{intensity matrix} of the urn defined in \eqref{A} below, and in
particular on the two largest (in real part) eigenvalues $\gl_1$ and
$\gl_2$.
If $\Re\gl_2\le\frac12\gl_1$ (a  \emph{small urn}), 
then, under some assumptions (including some version of irreducibility), 
the number of balls of a given colour is asymptotically normal,
while if $\Re\gl_2>\frac12\gl_1$ (a \emph{large urn}), 
then this is not true: 
there are (again under some assumptions, and after suitable normalization)
limits in distribution, but the limiting distributions
have no simple description and are (typically, at least) not normal;
furthermore, there may be 
oscillations so that suitable subsequences converge in
distribution but not the full sequence.
Another difference is that for a small urn, the limit is independent of 
the initial state, and therefore independent of
what
happens in any fixed finite set of draws (i.e., the limit theorem is mixing,
see \cite[Proposition 2]{AldousEagleson-stable}),
while for a large urn, on the contrary, 
there is an almost sure (\as) limit result
and thus 
the limit is essentially determined by what happens early in the process.

For large urns, \citet{P} proved,
assuming that the urn is \emph{balanced}, see \refS{Spolya}, 
a limit theorem which shows such an \as{} result and also shows 
convergence in $L^p$ for any $p$, 
and thus convergence of all moments.

For small urns, however, less has been known about moment convergence in
general.
Balanced deterministic urns with 2 colours were considered already by
\citet{Bernstein1,Bernstein2}, who showed both the
asymptotic normality in the small urn case and gave results on mean and
variance; \citet{Savkevitch} also considered this urn and studied
the mean and variance and, moreover, the third and fourth moments.
\citet{BagchiPal} (independently, but 45 years later)
gave another proof of asymptotic normality for balanced
deterministic small urns with 2 colours, by using the method of moments and thus
proving moment convergence as part of the proof.
\citet{BaiHu,BaiHu2005}, who consider an arbitrary number of colours and
allow random
replacements (under somewhat different conditions than ours, allowing also
time-dependent replacements), 
show asymptotic results for the mean and
(co)variance as part of their proofs of asymptotic
normality for small urns, 
using the same decomposition of the covariance matrix  as in the
present paper; 
however, these results are hidden inside the proof and are not
stated explicitly. 
More recently, \citet{SJ310} proved explicit results on asymptotics of
mean and variance, and also higher central moments of arbitrary order,
for irreducible small urns; the method there combined the known result on
asymptotic normality  in this case, and moment estimates by the method in
\cite{P} leading to uniform integrability. 

Note that asymptotics of mean and
variance often are of
interest in applications, complementing results on convergence in
distribution. (In some applications, results on mean and variance have been
proved separately by other methods.)
Moreover, although not surprising, it is satisfying to know that in a case
where a central limit theorem holds, also the mean and variance converge as
suggested by this central limit theorem. 
In particular, loosely speaking, the variance of the number of balls of
a given colour is asymptotic to the asymptotic variance.
Furthermore, for this random number,
the standard normalization by mean and standard deviation yields
convergence to a standard normal distribution.

The main purpose of the present paper is to 
give explicit asymptotics for
the first and second moments for a balanced small urn by an elementary direct
method. (Our assumptions are somewhat weaker than in \cite{SJ310}; we do not
assume that a central limit theorem holds, and 
also some non-normal cases are included, see \refR{Rtri}.)
We also include a simple result on non-degeneracy of the limit (\refT{TS}).
Precise statements are given in \refS{Sresults}.
Some results (\eg{} \refT{TE} and the lemmas in \refSs{Smatrix} and \ref{Spf2})
apply also to large urns.

Our method is closely related to the one used by \eg{} \citet{BaiHu,BaiHu2005};
it is also related to the method by \citet{P} and \cite{SJ310}, but
substantially simpler. 
The main idea (which has been used in various forms earlier, in particular
by \citet{BaiHu,BaiHu2005}) is that the
drawing of a ball and the subsequent addition of a set of balls, at time
$k$, say,
influences the composition of the urn at a later time $n$
not only directly by the added balls, but also indirectly since the added
balls change the probabilities for later draws. By including the expectation
of these later indirect effects, we find the real effect at time $n$ of the
draw at time $k$, and we may write the composition at time $n$ as the sum,
for $k\le n$, of these contributions, see \eqref{kia}.
The contributions for different $k$
are orthogonal, and thus the variance can be found by summing the variances of
these contribution.

Another purpose of the present paper is to demonstrate this elementary
method in detail and how can be used to obtain important results.
We believe that although the method is closely related to earlier proofs in
\eg{} \cite{BaiHu,BaiHu2005}, making the decomposition
\eqref{kia} explicit illuminates both the proofs and the general behaviour of 
\Polya{} urns, in particular the difference between large and small urns.
See the comments in \refS{Sfurther}.

\refS{Spolya} gives definitions and introduces the notation.
\refS{Sresults} contains the statements of the main results, which are proved
in Sections \ref{Spf1}--\ref{Spf2}.
\refS{Sex} presents some applications,
and \refS{Sfurther} contains some further comments on 
where the variance comes from, \ie, which draws are most important,
and the difference between small and large urns.
The appendices give some further, more technical, results.

\begin{remark}
We consider in the present paper only the mean and (co)\-variance.
As said above, similar results on convergence of higher moments for balanced
small urns are given in \cite{SJ310}
(under somewhat more restrictive assumptions than in the present paper),
by another method, based on showing uniform integrability. 
(An anonymous referee of the first version of the present paper has
suggested another, simpler, way to show uniform integrability; this can be
used to simplify the proofs in \cite{SJ310}.)

It is possible that the method in
the present 
paper can be extended to handle higher moments too, but we do not see any
immediate extension.
On the other hand, for the first and second moments,
the present method seems simpler, and perhaps also more informative, than
the one in \cite{SJ310}.
\end{remark}

\begin{problem}
In the present paper,
we consider only balanced urns.
We leave it as a challenging open problem to prove (or disprove?)
similar results for non-balanced urns.
\end{problem}

\section{\Polya{} urns}\label{Spolya}

\subsection{Definition and assumptions}

A (generalized) \Polya{} urn process is defined as follows.
(See \eg{} \citet{Mahmoud}, \citet{JKotz},
\citet{SJ154}, \citet{Flajolet-analytic} and \citet{P} for the history and
further references, as well as some different methods used to study such urns.)
There are balls of $q$ colours (types) $1,\dots, q$, where $2\le q<\infty$.
The composition of the urn at time $n$ is given by the vector
$X_n=(X_{n1},\dots,X_{nq})\in \ooo^q$, 
where $X_{ni} $ is the number of balls of colour $i$. 
The urn starts with a given vector $X_0$, 
and evolves according to a discrete time Markov process. 
Each colour $i$ has an \emph{activity} (or weight) $a_i\ge0$,
and a (generally random) 
\emph{replacement vector} $\xi_i=(\xi_{i1},\dots, \xi_{iq})$.
At each time $ n+1\ge 1 $, 
the urn is updated by drawing
one ball at random from the urn, with the probability of any ball
proportional to its activity. Thus, the drawn ball has colour $i$  
with probability 
\begin{equation}
\label{urn}
 \frac{a_iX_{ni}}{\sum_{j}a_jX_{nj}}. 
\end{equation}
If the drawn ball has type $i$, it is replaced
together with $\Delta X_{nj}$ balls of type $j$, $j=1,\dots,n$,
where the random vector 
$ \Delta X_{n}=(\Delta X_{n1},\dots, \Delta X_{nq}) $
has the same distribution as $ \xi_i$ and is
independent of everything else that has happened so far.  
Thus, the urn is updated to $X_{n+1}=X_n+\gD X_n$.

In many applications, the numbers $X_{nj}$ and $\xi_{ij}$ are integers, but
that is not necessary;
it has been noted several times that the \Polya{} urn process
is well-defined also for \emph{real} $X_{ni}$ and $\xi_{ij}$, 
with probabilities for the different replacements still given by \eqref{urn},
see \eg{} \cite{BaiHu}, 
\cite[p.~126]{BaiHuR}, 
\cite[Remark 4.2]{SJ154}, \cite[Remark 1.11]{SJ169} and
\cite{P}, and earlier \cite{Jirina} for the related case of branching
processes; the ``number of balls'' $X_{ni}$ may thus be any nonnegative real
number. (This can be interpreted as the amount (mass) of colour $i$ in the
urn, rather than the number of discrete balls.) 
The replacements $\xi_{ij}$ are thus random real numbers. 
We allow them to be negative,
meaning that balls may be subtracted from the urn.
However, we always assume that $X_0$ and the random vectors $\xi_i$ are
such that,
for every $n\ge0$,  \as
\begin{equation}\label{ten}
\text{each }
 X_{ni}\ge0 \quad \text{and}\quad  \sum_i a_i X_{ni}>0, 
\end{equation}
so that
\eqref{urn} really gives meaningful probabilities
(and the process does not stop due to lack of balls to be removed). 
An urn with such initial
conditions and replacement rules is called \emph{tenable}.

\begin{remark}\label{Rten}
A sufficient condition for tenability, which often is assumed in other
papers (sometimes with simple modifications), is that 
all $\xi_{ij}$ and $X_{0i}$ are integers
with
$\xi_{ij}\ge0$ for $j\neq i$ and $\xi_{ii}\ge -1$
(this means that we may remove the drawn ball but no other ball),
and furthermore, for example,   
$\sum_ja_j\xi_{ij}\ge0$ \as{}
(meaning that the total activity never decreases);
then
the urn is tenable for any $X_0$ with non-zero activity.
This is satisfied in most  applications we know of, but not all; see
\refR{Rgaps-types} for a different example. 
We shall \emph{not} assume this condition in the present paper, unless
explicitly said so.
\end{remark}

\begin{remark}
  \label{Rten2}
In all applications that we know of, each $\xi_i$ is a discrete random
vector, \ie{} it takes only a countable (usually a finite) number of
different values. This is not necessary, however; the results below hold
also if, e.g., some $\xi_{ij}$ is continuous. 
\end{remark}

We assume, for simplicity, that the initial composition $X_0$ is deterministic. 

\begin{remark}
The results
are easily extended to the case of random $X_0$ by conditioning on $X_0$,
but that may require some extra conditions or minor modifications in some of the
statements, which we leave to the reader.  
\end{remark}

The \Polya{} urn is \emph{balanced} if 
\begin{equation}\label{balanced}
   \sum_j a_j\xi_{ij} = b>0
\end{equation}
(\as)
for some constant $b$ and every $i$. In other words, the added activity
after each draw is fixed (non-random and not depending on the colour of the
drawn ball). This implies that the 
denominator in \eqref{urn} 
(which is the total activity in the urn)
is deterministic for each $n$, see \eqref{wn1}.
This is a significant simplification, and is assumed in many papers on
\Polya{} urns. 
(One exception is \cite{SJ154}, 
which is based on embedding in a continuous time 
branching process and stopping at a suitable stopping time, 
following \cite{AthreyaKarlin};
this method does not seem to easily give information on moments and is not
used in the present paper.)

\begin{remark}
We exclude the case $b=0$, which is quite different;
a typical example is a Markov chain, 
regarded as an urn always containing a single ball.
\end{remark}

We shall assume that the urn is tenable and balanced;
this is sometimes repeated for emphasis.

We also assume \eqref{gl1b} below; as discussed in \refR{Rgl1b} and
\refApp{Aten}, this is a 
very weak assumption needed to exclude some trivial cases allowed by our
definition of tenable; by \refL{Lapp} it is sufficient to assume that every
colour in the specification actually may occur in the urn, which always can
be achieved by eliminating any redundant colours.

Finally, in order to obtain moment results,
we assume that the replacements have second moments:
\begin{equation}\label{Exi2}
  \E \xi_{ij}^2<\infty,\qquad i,j=1,\dots,q.
\end{equation}
It follows that every $X_n$ has second moments, so the covariance matrix
$\Var(X_n)$ is finite for each $n$.

\begin{remark}
  In the tenable and balanced case, 
the assumption \eqref{Exi2} is almost redundant.
First, although there might be negative values of $\xi_{ij}$, 
we assume that the urn is tenable. Hence, given any
instance $(x_1,\dots,x_q)$ of the urn that may occur with positive probability
as some $X_n$, we have $\xi_{ij}\ge -x_j$ \as{}
for every $i$ and $j$ such that $a_ix_i>0$. 
In particular, if every $a_i>0$, and every colour may appear in the urn,
then each $\xi_{ij}$ is bounded below. 
Furthermore, still assuming $a_i>0$ for each $i$, this and \eqref{balanced}
implies that each $\xi_{ij}$ also is bounded above; hence $\xi_{ij}$ is
bounded and has moments of any order.
\end{remark}

\subsection{Notation}
We regard all vectors as column vectors.
We use standard notations for (real or complex)
vectors and matrices (of sizes $q$ and
$q\times q$, respectively);
in particular ${}'$ for transpose,
${}^*$ for Hermitean conjugate and
$\cdot$ for the standard scalar product; thus $u\cdot v=u'v$ for any
vectors $u,v\in\bbR^q$. 
We let $\norm{\,}$ denote the standard Euclidean norm for vectors, and 
the operator norm (or any other convenient norm) 
for matrices. 

Let $a:=(a_1,\dots,a_q)'$ be the vector of activities.
Thus, the balance condition \eqref{balanced} can be written $a\cdot\xi_i=b$.

The \emph{intensity matrix} of the \Polya{} urn is
the $ q\times q $  matrix
\begin{equation}\label{A}
A:=(a_j\E\xi_{ji})_{i,j=1}^{q}.   
\end{equation}
(Note that, for convenience and following \cite{SJ154}, we have defined $A$
so that the element $(A)_{ij}$ is a measure of the intensity of adding balls
of colour $i$ coming from drawn balls of colour $j$; the transpose matrix
$A'$ is often used in other papers.)
The intensity matrix $ A $ with its eigenvalues and
eigenvectors has a central role for  asymptotical results.

Let $\gs(A)$ (the \emph{spectrum} of $A$) be the set of eigenvalues of $A$.

We shall use the Jordan decomposition of the matrix $A$ in the
following form. 
There exists a decomposition of the {complex} space $\bbC^q$ as a
direct sum $\bigoplus_\gl \el$ of generalized eigenspaces $\el$, such
that $A-\gl I$ is a nilpotent operator on $\el$;
here $\gl$ ranges over the set $\gs(A)$ of eigenvalues of $A$.
($I$ is the identity matrix of appropriate size.)
In other words, there exist projections $\pl$, $\gl\in\gs(A)$, 
that commute with $A$ and satisfy
\begin{gather}
\sum_{\gl\in\gsa}\pl=I,
\label{pl}\\
A\pl=\pl A=\gl\pl+\nl,
\label{24b}
\end{gather}
where $\nl=\pl\nl=\nl\pl$ is nilpotent. 
Moreover, $\pl P_\mu=0$ when $\gl\neq\mu$.
We let 
$\nul\ge0$ be the integer such that $\nl^{\nul}\neq0$ but $\nl^{\nul+1}=0$.
(Equivalently, in the Jordan normal form of $A$, the largest Jordan
block with $\gl$ on the diagonal has size $\nul+1$.) 
Hence $\nul=0$ if and only if $\nl=0$, and this happens for all $\gl$
if and only if $A$ is diagonalizable, \ie{} if and only 
if $A$ has a complete set of $q$ linearly independent eigenvectors.
(In the sequel, $\gl$ will always denote an eigenvalue. We may for
completeness define $\pl=\nl=0$ for every $\gl\notin\gsa$.)

The eigenvalues of $A$ are denoted $\gl_1,\dots,\gl_q$ (repeated according
to their algebraic multiplicities); we assume that they are ordered 
with decreasing real parts: $\Re\gl_1\ge\Re\gl_2\ge\dots$, and furthermore,
when the real parts are equal, in order of decreasing $\nu_j:=\nu_{\gl_j}$. 
In particular, 
if $\gl_1>\Re\gl_2$, then $\nu_j\le\nu_2$ for every eigenvalue $\gl_j$
with $\Re\gl_j=\Re\gl_2$.

Recall that the urn is called \emph{small} if $\Re\gl_2\le\frac12\gl_1$
and \emph{large} if $\Re\gl_2>\frac12\gl_1$; 
the urn is \emph{strictly small}
if $\Re\gl_2<\frac12\gl_1$.

In the balanced case, by \eqref{A} and \eqref{balanced},
\begin{equation}\label{bala}
  a'A
=\Bigpar{\sumiq a_i(A)_{ij}}_j
=\Bigpar{\sumiq a_ia_j\E \xi_{ji}}_j
=\bigpar{a_j\E (a\cdot \xi_{j})}_j
=ba',
\end{equation}
\ie, $a'$ is a left eigenvector of $A$ with eigenvalue $b$.
Thus $b\in\gsa$.
We shall assume that, moreover, $b$ is the largest eigenvalue, \ie,
\begin{equation}\label{gl1b}
  \gl_1=b.
\end{equation}

\begin{remark}\label{Rgl1b}
In fact, \eqref{gl1b} is a very weak assumption.
For example,
if each $\xi_{ij}\ge0$, then $A$ is a matrix with non-negative elements, and
since the eigenvector $a'$ is non-negative,
\eqref{gl1b} is a consequence of
the Perron--Frobenius theorem.
The same holds (by considering $A+cI$ for a suitable $c>0$) under the
assumption in \refR{Rten}. Under our, more general, definition of tenability,
there are counterexamples, see \refE{Ebad},
but we show in \refApp{Aten} that they are so only in a trivial way, and
that we may assume \eqref{gl1b} without real loss of generality.
(Of course, the proof of \refL{Lapp}, which uses \refL{L1}, does not use the
assumption \eqref{gl1b}.)
\end{remark}

We shall in our theorems furthermore assume that
$\Re\gl_2<\gl_1$ (and often more), and thus that
$\gl_1=b$ is a simple eigenvalue.
There are thus corresponding left and right eigenvectors
$u_1'$ and $v_1$ that are unique up to normalization.
By \eqref{bala}, we may choose $u_1=a$. Furthermore, we let $v_1$ be
normalized by
\begin{align}
\label{normalized}
u_1\cdot v_1=a\cdot v_1=1.
\end{align}
Then the projection $P_{\gl_1}$ is given by
\begin{equation}
  \label{pl1}
P_{\gl_1}=v_1u_1'.
\end{equation}
Consequently, 
in the balanced case, for any vector $v\in\bbR^q$,
\begin{equation}\label{p1}
  P_{\gl_1}v=v_1u_1'v=v_1a'v=(a\cdot v)v_1.
\end{equation}

\begin{remark}\label{Rmulti}
The dominant eigenvalue
$\gl_1$ is simple, and $\Re\gl_2<\gl_1$ if, for example,
the matrix $A$ is irreducible, but not in general. A simple counterexample
is the original \Polya{} urn, see 
\citet{Markov1917},
\citet{EggPol} and \citet{Polya} (with $q=2$),
where each ball is replaced together with $b$ balls of the same colour (and
every $a_i=1$); then $A=bI$ and $\gl_1=\dots=\gl_q=b$.
As is well-known, the asymptotic behaviour in this case is quite different in
this case; in particular, $X_n/n$ converges in distribution to a
non-degenerate distribution and not to a constant, see \eg{} \cite{Polya} and
\cite{JKotz}. 
\end{remark}

Define also
\begin{equation}
  \PIx:=\sumglx P_\gl = I-P_{\gl_1},
\end{equation}

Furthermore, define the symmetric matrix
\begin{equation}\label{B}
  B:=\sumiq a_i v_{1i}\E\bigpar{\xi_i\xi_i'}
\end{equation}
and, 
if the urn is \emph{strictly small}, 
noting that $\PIx$ commutes with $e^{sA}:=\sumko (sA)^k/k!$,
\begin{equation}\label{gSI}
  \gS_I:=\intoo \PIx e^{sA} B e^{sA'} \PIx' e^{-\gl_1 s} \dd s.
\end{equation}
This integral converges absolutely when the urn is strictly small, 
as can be seen from the proof of \refT{TV1}, or directly
because 
$\norm{ \PIx e^{sA}}=O\bigpar{s^{\nu_2}e^{\Re\gl_2s}}$ for $s\ge1$, 
as is easily seen from
\refL{LT}. 
(The integral is matrix-valued; 
the space of $q\times q$ matrices is a finite-dimensional space and the
integral can be interpreted component-wise.)
See also \refApp{SSnote}.

Unspecified limits are as \ntoo. As usual, $a_n=O(b_n)$ means that $a_n/b_n$
is bounded; 
here $a_n$ may be vectors or matrices and $b_n$ may be complex numbers;
we do not insist that $b_n$ is positive.

$\ceil{x}$ is the smallest integer $\ge x$.

\section{Main results}\label{Sresults}

Our main results on asymptotics of mean and variance are the following.
Proofs are given in \refS{Spf2}.
As said in the introduction, 
results of this type exist, mainly implicitly, in earlier work.
In particular, under similar (but not identical) assumptions,
\refT{TE} is implicit in \citet[Theorem 3.2]{BaiHu2005} and
its proof, and explicit
in \citet[Proposition 7.1]{P};
\refTs{TV1} and \ref{TV2} are implicit in \citet{BaiHu,BaiHu2005}.
\begin{theorem}
  \label{TE}
If the \Polya{} urn is tenable, balanced and
$\Re\gl_2<\gl_1$, then, for $n\ge2$,
\begin{equation}\label{te}
  \begin{split}
\E X_n &= (n\gl_1 + a\cdot X_0)v_1 + O\bigpar{n^{\Re\gl_2/\gl_1}\log^{\nu_2} n}
\\
&= n \gl_1 v_1 + O\bigpar{n^{\Re\gl_2/\gl_1}\log^{\nu_2} n+1}	
\\
&= n \gl_1 v_1 + o(n).
  \end{split}
\end{equation}
In particular, if the urn is strictly small, 
\ie{} $\Re\gl_2<\frac12\gl_1$, then
\begin{equation}\label{tea}
  \E X_n = n\gl_1 v_1 + o\bigpar{n\qq}.
\end{equation}
\end{theorem}

\begin{theorem}
  \label{TV1}
If the \Polya{} urn is tenable, balanced and
strictly small, 
\ie{} $\Re\gl_2<\frac12\gl_1$, then
\begin{equation}\label{tv1}
 n\qw \Var(X_n) \to 
\gS:=
\gl_1\gS_I.
\end{equation}
\end{theorem}

\begin{theorem}
  \label{TV2}
If the \Polya{} urn is tenable, balanced and
small but not strictly small, 
\ie{}  $\Re\gl_2=\frac12\gl_1$, then
\begin{equation*}
 (n\log^{2\nu_2+1}n)\qw \Var(X_n) \to 
 \frac{\gl_1^{-2\nu_2}}{(2\nu_2+1)(\nu_2!)^2}
\sum_{\Re\gl=\frac12\gl_1}
N_\gl^{\nu_2}P_\gl {B} P_\gl^*\xpar{N_\gl^*}^{\nu_2}
.
\end{equation*}
\end{theorem}

\begin{remark}\label{RTV}
  Under some additional assumptions (irreducibility of $A$, at least if we
  ignore colours with activity 0, and, for example,
  the condition in \refR{Rten}), 
\cite[Theorems 3.22--3.23 and Lemma 5.4]{SJ154} show
that if the urn is small, 
then $X_n$ is asymptotically normal, with the asymptotic covariance matrix
equal to the limit in Theorem \ref{TV1} ($\Re\gl_2<\frac12\gl_1$) or
or \refT{TV2} ($\Re\gl_2=\frac12\gl_1$). For example, in the strictly small
case,
$n\qqw(X_n-n\gl_1v_1)\dto N(0,\gS)$.
Hence (under these hypotheses), Theorems \ref{TE}--\ref{TV2} can be
summarized by saying that the mean and (co)variances converge as expected in
these central limit theorems.
\end{remark}

We also obtain the following version of the law of large
numbers for \Polya{} urns. 
Convergence \as{} has been shown before under various assumptions
(including the unbalanced case as well as time-inhomogenous
generalizations), 
see \cite[Section V.9.3]{AN}, 
\cite{AldousFP}, 
\cite{BaiHuR}, 
\cite[Theorem 2.2]{HuZhang2004a},
\cite[Theorem 3.21]{SJ154}, 
\cite[Theorem 2.2]{BaiHu2005}
and is included here for
completeness and because our conditions are somewhat more general; 
the $L^2$ result is in
\cite[Remark 7.1(2)]{P}.

\begin{theorem}\label{TL2}
If the \Polya{} urn is tenable, balanced and
$\Re\gl_2<\gl_1$, then, as \ntoo,
$X_n/n\to \gl_1v_1$ \as{} and in $L^2$. 
\end{theorem}

The asymptotic covariance matrix $\gS$ in \eqref{tv1} is always singular,
since, by \eqref{gSI}, $\gS=\PIx\gS\PIx'$ and thus 
$u'\gS u=u'\PIx\gS\PIx'u=0$ when $\PIx'u=0$, which happens when
$P_{\gl_1}' u=u$, \ie, when $u$ is a multiple of the left eigenvector $u_1=a$.
In the balanced case, this is easy to see: 
$a\cdot X_n$ is deterministic and thus $\Var(a\cdot X_n)=0$; hence 
$a'\gS a=0$ 
since for any vector $u$, by \eqref{tv1},
\begin{equation}\label{agnus}
n\qw \Var(u\cdot X_n)=n\qw u'\Var(X_n) u \to u'\gS u.
\end{equation}
With an extra assumption, this is the only case when the asymptotic variance
$u'\gS u$ vanishes (cf.\ \cite[Remark 3.19]{SJ154}).
Let $\tA$ be the submatrix of $A$ obtained by deleting all rows and columns
corresponding to colours with activity $a_i=0$.

\begin{theorem}  \label{TS}
Suppose that the \Polya{} urn is tenable, balanced and strictly small, 
\ie{} $\Re\gl_2<\frac12\gl_1$, 
and, furthermore, that 
$\tA$ is irreducible.
If $u\in \bbR^q$, then 
$u'\gS u =0$ if and only if for every $n\ge0$, $\Var(u\cdot X_n)=0$, \ie,
$u\cdot X_n$ is deterministic.
\end{theorem}

\begin{remark}\label{Rtri}
If $\tA$ is reducible, then, on the contrary,
$\gS$ is typically more singular. 
As an extreme example, consider a ``triangular'' urn
with two colours, activities $a_i=1$ and deterministic replacements
$\xi_1=(1,0)$, $\xi_2=(1-\gl,\gl)$
for a real $\gl\in(0,1)$. (Starting with one ball of each colour, say.)
Then $A=
\smatrixx{1&1-\gl\\0&\gl}
$.
The eigenvalues are $1$ and $\gl$, so the urn is
strictly small if $\gl<\frac12$. However, $v_1=(1,0)$, and thus \eqref{B} yields
$B=\xi_1\xi_1'=v_1v_1'$, and thus by \eqref{PBP} (or a direct calculation)
$\PIx B =0$, and thus $\gS=\gS_I=0$. 
Theorems \ref{TV1} and \ref{TV2} are still valid, but say only that the
limit is 0. 
In fact, in this example, the proper normalization is $n^\gl$: it follows
from \cite[Theorem 1.3(v)]{SJ169} that
$n^{-\gl}X_{n2}=n^{-\gl}(n+2-X_{n1})\dto W$ for some non-degenerate (and
non-normal) random variable $W$.
Moreover, calculations similar to those in \refS{Spf2} show that 
$\E X_{n2}\sim c_1 n^\gl$ and $\Var X_{n2}\sim c_2 n^{2\gl}$ for some
$c_1,c_2>0$, as shown earlier in
\cite[Example 7.2(2)]{P}. 
\end{remark}

\begin{remark}
It is easily seen that $\tA$ is irreducible if and only if 
$v_{1i}>0$ for every $i$ with $a_i>0$.
\end{remark}

\section{Proofs, first steps}\label{Spf1}
Let $I_n$ be the colour of the $n$-th drawn ball, and let
\begin{equation}\label{gDX}
  \gD X_n:=X_{n+1}-X_n
\end{equation}
and
\begin{equation}\label{wn}
  w_n:=a\cdot X_n,
\end{equation}
the total weight (activity) of the urn.
Furthermore, let $\cF_n$ be the $\gs$-field generated by $X_1,\dots,X_n$.
Then, by the definition of the urn,
\begin{equation}\label{pin}
  \P\bigpar{I_{n+1}=j\mid \cF_n} = \frac{a_jX_{nj}}{w_n}
\end{equation}
and, consequently, recalling \eqref{A},
\begin{equation}\label{mia}
  \begin{split}
\E\bigpar{\gD X_n\mid \cF_n}
&=\sumjq \P\bigpar{I_{n+1}=j\mid \cF_n}\E\xi_j	
=\frac{1}{w_n}\sumjq a_j X_{nj}\E\xi_j	
\\&
=\frac{1}{w_n}\Bigpar{\sumjq (A)_{ij} X_{nj}}_i
=\frac{1}{w_n} A X_{n}.
  \end{split}
\end{equation}

Define
\begin{equation}\label{Yn}
  Y_n:=\gD X_{n-1} - \E\bigpar{\gD X_{n-1}\mid \cF_{n-1}}.
\end{equation}
Then, $Y_n$ is $\cF_n$-measurable and, obviously,
\begin{equation}
  \label{eyn}
\E\bigpar{Y_n\mid\cF_{n-1}}=0
\end{equation}
and, by \eqref{gDX}, \eqref{Yn} and \eqref{mia},
\begin{equation}
  X_{n+1}=X_n+Y_{n+1}+w_n\qw A X_n
=\bigpar{I+w_n\qw A}X_n+Y_{n+1}.
\end{equation}
Consequently, by induction, for any $n\ge0$,
\begin{equation}\label{tia}
  X_n=\prod_{k=0}^{n-1}\bigpar{I+w_k\qw A}X_0
  +\sum_{\ell=1}^{n}\prod_{k=\ell}^{n-1}\bigpar{I+w_k\qw A}Y_\ell,
\end{equation}
where (as below) an empty matrix product is interpreted as $I$.

We now use the assumption that the urn is balanced, so $a\cdot\gD X_n=b$ and
thus 
by \eqref{gDX}--\eqref{wn},
$w_n$ is deterministic with
\begin{equation}\label{wn1}
  w_n=w_0+nb,
\end{equation}
where the initial weight $w_0=a\cdot X_0$.
We define the matrix products
\begin{equation}\label{qia}
  F_{i,j}:=\prod_{i\le k<j}\bigpar{I+w_k\qw A},
\qquad 0\le i\le j,
\end{equation}
and write  \eqref{tia} as
\begin{equation}
  \label{kia}
  X_n=F_{0,n} X_0+\sum_{\ell=1}^{n}F_{\ell,n} Y_{\ell}.
\end{equation}
As said in the introduction, we can regard the term $F_{\ell,n}Y_\ell$ as
the real effect on $X_n$ of the $\ell$-th draw, including the expected later
indirect effects.

Taking the expectation we find, since $\E Y_\ell=0$ by \eqref{eyn}, 
and the $F_{i,j}$ and $X_0$ are nonrandom,
\begin{equation}
  \label{EX}
\E X_n=F_{0,n} X_0.
\end{equation}
Hence, \eqref{kia} can also be written
\begin{equation}\label{gw}
  X_n-\E X_n =\sum_{\ell=1}^{n}F_{\ell,n} Y_\ell.
\end{equation}
Consequently, the covariance matrix can be computed as
\begin{equation}\label{dia}
  \begin{split}
\Var(X_n)
&:=	\E\bigpar{(X_n-\E X_n)(X_n-\E X_n)'}
\\&\phantom:
=\E\sum_{i=1}^{n}\sum_{j=1}^{n}
  \bigpar{F_{i,n} Y_i}\bigpar{F_{j,n} Y_j}'
\\&\phantom:
=\sum_{i=1}^{n}\sum_{j=1}^{n} F_{i,n} \E\bigpar{Y_i Y_j'} F_{j,n}'
.  \end{split}
\end{equation}
However, if $i>j$, then 
$\E\bigpar{Y_i\mid \cF_j}=0$ by \eqref{eyn},
and since $Y_j$ is $\cF_{j}$-measurable, we have
\begin{equation}
  \E\bigpar{Y_iY_j'} =
\E\bigpar{\E (Y_i\mid \cF_j)Y_j'}=0.
\end{equation}
Taking the transpose we see that $\E\bigpar{Y_iY_j'}=0$ also when $i<j$.
Hence, all nondiagonal terms vanish in \eqref{dia}, and we find
\begin{equation}\label{varx}
  \begin{split}
\Var(X_n)
=\sum_{i=1}^{n} F_{i,n} \E\bigpar{Y_i Y_i'} F_{i,n}'
.  \end{split}
\end{equation}

The formulas \eqref{EX} and \eqref{varx} form the basis of our proofs, and it
remains mainly to analyse the matrix products $F_{i,j}$.

\begin{remark}
  The formula \eqref{tia} holds for general \Polya{} urns, also when they
  are not balanced. However, in the general case, the total weights $w_k$
  are random, and they are dependent on each other and on the $Y_\ell$, and
it seems difficult to draw any useful consequences from \eqref{tia}; certainly
the arguments above fail because the $F_{i,j}$ would be random.
\end{remark}

\begin{remark}
  As remarked by a referee, since $P_{\gl_1}Y_\ell=0$ by
  \refL{LY0}, 
we may also write 
\eqref{gw} as
\begin{equation}\label{gw^}
  X_n-\E X_n =\sum_{\ell=1}^{n} F_{\ell,n}\PIx Y_\ell
=\sum_{\ell=1}^{n} \hF_{\ell,n} Y_\ell,
\end{equation}
where $\hF_{i,j}:=\PIx F_{i,j}=F_{i,j}\PIx=\prod_{i\le k<j}(I+w_k\qw\hA)-P_{\gl_1}$
with $\hA:=\PIx A$.
This could be used instead of \eqref{gw}
to make another version of the proofs below;
the two versions are very similar and essentially equivalent.
(See \cite{HuZhang2004a} for a version essentially of this type.)
The form \eqref{gw^} has the advantage that we have eliminated the (large)
deterministic part corresponding to $P_{\gl_1}$; for example, assuming
$b=1$, we obtain for $1\le\ell\le n$,
$\norm{\hF_{\ell,n}}=O\bigpar{(n/\ell)^{\Re\gl_2}(1+\log(n/\ell))^{\nu_2}}$,
see \refL{L2}. Nevertheless, we prefer to use $F_{i,j}$ in the proofs below.
\end{remark}

\section{Estimates of matrix functions}\label{Smatrix}
In this section we derive some estimates of $F_{i,n}$; these are used in the
next section together with \eqref{varx} to obtain the variance asymptotics.
The estimates of $F_{i,n}$ are obtained by standard matrix calculus
including a Jordan decomposition of $A$.
Similar estimates, have been used in several related
papers, \eg{}
\cite{HuZhang2004a}, 
\cite{BaiHu2005}, 
\cite{ZhangHuCheung},
\cite{SJ154}. 
For completeness we nevertheless give detailed proofs. 

For notational convenience, we make
from now on the simplifying assumption
$b=1$. 
(For emphasis and clarity, we repeat this assumptions in some statements; 
it will always be in force, whether stated or not.)
This is no loss of generality; we can divide all activities by
$b$ and let the new activities be $\hat a:=a/b$; this defines the same
random evolution of the urn and we 
have $\hat a\cdot \xi_{i}=b/b=1$  for every $i$,
so the modified urn is also balanced, with balance $\hat b=1$.
Furthermore, the intensity matrix $A$ in \eqref{A} is divided by $b$, so all
eigenvalues $\gl_i$ are divided by $b$, but their ratios remain the same;
the projections $P_{\gl}$ remain the same
while the nilpotent parts $N_\gl$ are divided by $b$,
and in both cases the indices are shifted;
also, with the normalization \eqref{normalized},
$u_1=a$ is divided by $b$ while $v_1$ is multiplied by $b$.
It is now easy to check that $\gl_1v_1$, $B$ and $\gl_1\gS_I$ 
are invariant,
and thus the theorems all follow from the special case
$b=1$. 
By the assumption \eqref{gl1b}, see \refR{Rgl1b} and \refApp{Aten}, we thus
have $\gl_1=1$.

Note that \eqref{wn1} now becomes
\begin{equation}\label{wn2}
w_n=n+w_0.  
\end{equation}

Note also that \eqref{qia} can be written
$F_{i,j}=f_{i,j}(A)$, where 
$0\le i\le j$ and $f_{i,j}$ is the polynomial
\begin{equation}\label{qib}
  \begin{split}
  f_{i,j}(z)&:=
\prod_{i\le k<j}\bigpar{1+w_k\qw z}
=\prod_{i\le k<j}\frac{w_k+ z}{w_k}
=\prod_{i\le k<j}\frac{k+w_0+ z}{k+w_0}
\\&\phantom:
= \frac{\gG(j+w_0+z)/\gG(i+w_0+z)}	{\gG(j+w_0)/\gG(i+w_0)}
= \frac{\gG(j+w_0+z)}{\gG(j+w_0)}	
\cdot \frac{\gG(i+w_0)}{\gG(i+w_0+z)}	
.  \end{split}
\end{equation}

Recall that by the functional calculus in spectral theory, see \eg{} 
\cite[Chapter VII.1--3]{DunfordS},
we can define $f(A)$ not only for polynomials $f(z)$ but for any function $f(z)$
that is analytic in a neighbourhood of the spectrum $\gs(A)$.
Furthermore, if $K$ is a compact set 
that contains $\gs(A)$ in its interior
(for example a sufficiently large disc), then there exists a constant $C$
(depending on $A$ and $K$)
such that for every $f$ analytic in a neighbourhood of $K$, 
\begin{equation}\label{nC}
  \norm{f(A)}\le C\sup_{z\in K} |f(z)|.
\end{equation}
We shall use the functional calculus 
mainly for polynomials and the entire functions
$z\mapsto t^z=e^{(\log t)z}$ for fixed $t>0$; in these cases, $f(A)$ can
be defined by a Taylor series expansion as we did before \eqref{gSI}.
Note also that the general theory applies to
operators in a Banach space; we only need the simpler finite-dimensional case
discussed in \cite[Chapter VII.1]{DunfordS}.

We shall use the following formula for $f(A)$, 
where $f\mm$ denotes the $m$-th derivative of $f$.
(The formula can be seen as a Taylor expansion, see the proof.) 
\begin{lemma}
  \label{LT}
For any entire function $f(\gl)$, and any $\gl\in\gs(A)$,
\begin{equation}
  \label{lt}
f(A)P_\gl = \sum_{m=0}^{\nu_\gl}\frac{1}{m!} f\mm(\gl)N_\gl^mP_\gl.
\end{equation}
\end{lemma}
\begin{proof}
This is a standard formula in the finite-dimensional case, see
  \cite[Theorem VII.1.8]{DunfordS}, but we give for completeness 
a simple (and perhaps informative) 
proof when $f$ is a polynomial (which is the only case
  that we use, and furthermore implies the general case by 
\cite[Theorem VII.1.5(d)]{DunfordS}). We then have the Taylor expansion
$f(\gl+z)=\sum_{m=0}^\infty \wmm f\mm(\gl)z^m$, which can be seen as an
algebraic identity for polynomials in $z$ (the sum is really finite since
$f\mm=0$ for large $m$), and thus
\begin{equation}
  f(A)P_\gl
= f(\gl I+N_\gl)P_\gl
=\sum_{m=0}^\infty \wmm f\mm(\gl)N_\gl^m P_\gl,
\end{equation}
where $N_\gl^m=0$ when $m>\nu_\gl$.
\end{proof}

Our strategy is to first show estimates for the polynomials $f_{i,j}(z)$
in \eqref{qib} and then use these together with \eqref{nC} and \eqref{lt} to
show the estimates for $F_{i,j}=f_{i,j}(A)$ that we need.

\begin{lemma}
  \label{L0}
  \begin{thmenumerate}
  \item 
For every fixed $i$, 
as $j\to\infty$,
\begin{equation}\label{l0a}
  f_{i,j}(z)=j^z \frac{\gG(i+w_0)}{\gG(i+w_0+z)}\bigpar{1+o(1)},
\end{equation}
\uniz.
\item 
As $i,j\to\infty$ with $i\le j$,
\begin{equation}\label{l0b}
  f_{i,j}(z)=j^z i^{-z}\bigpar{1+o(1)},
\end{equation}
\uniz.
  \end{thmenumerate}
\end{lemma}
\begin{proof}
Both parts follow from \eqref{qib} and the fact that
\begin{equation}\label{gg}
\frac{\gG(x+z)}{\gG(x)}=x^z\etto,   
\end{equation}
uniformly for $z$ in a compact set, as
$x\to\infty$ (with $x$ real, say), which is an easy and well-known
consequence of Stirling's formula, see \cite[5.11.12]{NIST}.
(Note that $\gG(i+w_0)/\gG(i+w_0+z)$ is an entire function for any $i\ge0$,
since $w_0>0$. $\gG(j+w_0+z)/\gG(j+w_0)$ has poles, but for $z$ in a fixed
compact set, this function is analytic when $j$ is large enough.)
\end{proof}

For the derivatives $f_{i,j}\mm(z)$ there are corresponding estimates.

\begin{lemma}
  \label{L0m}
Let $m\ge0$.
\begin{romenumerate}
\item \label{L0ma}
For every fixed $i\ge0$, as $\jtoo$,
\begin{equation}\label{l0ma}
  f\mm_{i,j}(z)
= {j}^z(\log j)^m  \frac{\gG(i+w_0)}{\gG(i+w_0+z)}
+
o\bigpar{{j}^z\log^m j},
\end{equation}
\uniz.

\item \label{L0mb}
As $i,j\to\infty$ with $i\le j$,
\begin{equation}\label{l0mb}
  f\mm_{i,j}(z)=
\Bigparfrac{j}{i}^z\Bigpar{\log\frac{j}{i}}^m
+ o\lrpar{\Bigparfrac{j}{i}^z\Bigpar{1+\log\frac{j}{i}}^m},
\end{equation}
\uniz.
\end{romenumerate}
\end{lemma}

\begin{proof}
  \pfitemref{L0ma}
Let $g_j(z)= j^{-z}f_{i,j}(z)$.
Then, by \eqref{l0a},
\begin{equation}\label{jeeves}
g_j(z)=
 \frac{\gG(i+w_0)}{\gG(i+w_0+z)}\bigpar{1+o(1)}
=O(1)
\qquad\text{as \jtoo},
\end{equation}
uniformly in each compact set,
and thus by Cauchy's estimates, for any $\ell\ge1$,
\begin{equation}\label{ask}
g_j\qll(z)=O(1)  
\qquad\text{as \jtoo},
\end{equation}
uniformly in each compact set.
By Leibnitz' rule, 
\begin{equation}
  \begin{split}
f_{i,j}\mm(z)
&=
\frac{\ddx^m}{\ddx z^m}\bigpar{j^z g_j(z)}  
=\sum_{\ell=0}^m 
\binom{m}{\ell}
\frac{\ddx^\ell}{\ddx z^\ell}j^z \cdot
g_j^{(m-\ell)}(z)
\\&
=\sum_{\ell=0}^m 
\binom{m}{\ell}
(\log j)^\ell j^z
g_j^{(m-\ell)}(z)	
  \end{split}
\end{equation}
and \eqref{l0ma} follows by \eqref{jeeves}--\eqref{ask}.

\pfitemref{L0mb}
Similarly, let (for $1\le i\le j$)
$h_{i,j}(z)= (i/j)^{z}f_{i,j}(z)$.
Then, by \eqref{l0b},
\begin{equation}\label{embla}
h_{i,j}(z)=
1+o(1)
\qquad\text{as $i,j\too$},
\end{equation}
uniformly in each compact set,
and thus by Cauchy's estimates, for any $\ell\ge1$,
\begin{equation}\label{hector}
h_{i,j}\qll(z)=
\frac{\ddx^\ell}{\ddx z^\ell}\bigpar{h_{i,j}(z)-1}
=o(1)  
\qquad\text{as $i,j\too$},
\end{equation}
uniformly in each compact set.
By Leibnitz' rule, 
\begin{equation}
  \begin{split}
f_{i,j}\mm(z)
&=
\frac{\ddx^m}{\ddx z^m}\bigpar{(j/i)^z h_{i,j}(z)}  
=\sum_{\ell=0}^m 
\binom{m}{\ell}
\frac{\ddx^\ell}{\ddx z^\ell}(j/i)^z \cdot h_{i,j}^{(m-\ell)}(z)
\\&
=\sum_{\ell=0}^m 
\binom{m}{\ell}
\bigpar{\log (j/i)}^\ell (j/i)^z
h_{i,j}^{(m-\ell)}(z)	
  \end{split}
\end{equation}
and \eqref{l0mb} follows by \eqref{embla}--\eqref{hector}.
\end{proof}

We now apply these estimates to $F_{i,j}$, noting that 
by \refL{LT},
\begin{equation}\label{ltij}
  \begin{split}
F_{i,j}P_\gl=f_{i,j}(A)P_\gl = 
\sum_{m=0}^{\nu_\gl}\frac{1}{m!} f_{i,j}\mm(\gl)N_\gl^mP_\gl.	
  \end{split}
\end{equation}

\begin{lemma}\label{L1}
If $b=1$, then, for $n\ge2$ and $\gl\in\gsa$,
  \begin{equation}\label{l1ax}
  {F_{0,n} P_\gl} 
=n^{\gl}\log^{\nu_\gl}n \frac{\gG(w_0)}{\nul!\,\gG(w_0+\gl)}\nl^{\nul}\pl
+ o\bigpar{n^{\Re\gl}\log^{\nu_\gl}n}.	
  \end{equation}
\end{lemma}

\begin{proof}
By \eqref{ltij} and \eqref{l0ma},
\begin{equation}
  \begin{split}
F_{0,n}P_\gl
&=\sum_{m=0}^{\nu_\gl}\frac{1}{m!} f_{0,n}\mm(\gl)N_\gl^mP_\gl
=\frac{1}{\nu_{\gl}!} f_{0,n}^{(\nul)}(\gl)N_\gl^{\nul} P_\gl
+\sum_{m=0}^{\nu_\gl-1} O\bigpar{n^\gl \log^m n},
  \end{split}
\end{equation}
which yields \eqref{l1ax} by another application of \eqref{l0ma}.
\end{proof}

\begin{lemma}\label{L2}
If $b=1$, then, for $1\le i\le j$ and $\gl\in\gsa$,
  \begin{equation}\label{l2}
  F_{i,j} P_\gl = O\bigpar{(j/i)^{\Re\gl}(1+\log(j/i))^{\nu_\gl}}.	
  \end{equation}
More precisely,
for any $\nu\ge\nu_\gl$, 
as $i,j\to\infty$ with $i\le j$,
\begin{multline}\label{tedeum}
F_{i,j} P_\gl = \frac{1}{\nu!}\parfrac{j}{i}^\gl  
\log^{\nu}\parfrac{j}{i} {N_\gl^{\nu}P_\gl }
+ o\lrpar{\Bigparfrac{j}{i}^{\Re\gl}  
\log^{\nu}\Bigparfrac{j}{i}}
\\
+
O\lrpar{
  \Bigparfrac{j}{i}^{\Re\gl}\Bigpar{1+\log^{\nu-1}\Bigparfrac{j}{i}}}.	    
\end{multline}
\end{lemma}
\begin{proof}
This is similar to the proof of \refL{L1}.
First, \eqref{l2} follows directly from \eqref{ltij} and \eqref{l0mb}.

For \eqref{tedeum}, note that the summation in \eqref{ltij} may be
extended to $m\le\nu$, since $N_\gl^m=0$ when $m>\nu_\gl$. 
Then use \eqref{l0mb} for each term $m=\nu$.
\end{proof}

\begin{lemma}
  \label{L1b}
If\/ $\Re\gl_2<\gl_1=b=1$,  then
for $0\le i\le j$,
\begin{equation}\label{l1bx}
  \begin{split}
F_{i,j}P_{\gl_1}=
f_{i,j}(\gl_1)P_{\gl_1}	
=\frac{j+w_0}{i+w_0}P_{\gl_1}.
  \end{split}
\end{equation}
\end{lemma}

\begin{proof}
Since $\gl_1$ thus is assumed to be a simple eigenvalue,
$\nu_{\gl_1}=0$. (Alternatively, see \refL{Lapp}.)
Hence, \eqref{ltij} yields 
$F_{i,j}\pli=f_{i,j}(\gl_1)\pli$.
Furthermore,  \eqref{qib} yields
\begin{equation}\label{qic}
  f_{i,j}(\gl_1)=f_{i,j}(1)=\frac{j+w_0}{i+w_0},
\end{equation}
and \eqref{l1bx} follows.
\end{proof}

\begin{lemma}\label{L3}
  For any fixed $x\in(0,1]$, as \ntoo,
\begin{equation}
  F_{\ceil{xn},n} \to x^{-A}.
\end{equation}
\end{lemma}
\begin{proof}
  Let $K$ be a compact set containing $\gs(A)$ in its interior.
As \ntoo, by \eqref{l0b}, 
\begin{equation}
  f_{\ceil{xn},n}(z)
=\Bigparfrac{n}{\ceil{xn}}^z \bigpar{1+o(1)}
=x^{-z}\etto=x^{-z}+o(1),
\end{equation}
uniformly for $z\in K$.
Consequently, 
$f_{\ceil{xn},n}(z)-x^{-z}\to0$ uniformly on $K$, and thus 
$F_{\ceil{xn},n}-x^{-A}\to0$ by \eqref{nC}.
\end{proof}

\begin{lemma}
  \label{LB}
There exists $i_0$ and $C$ such that if $i_0\le i\le j\le 2i$, then
$\norm{F_{i,j}\qw}\le C$. 
\end{lemma}
\begin{proof}
  Let again $K$ be a compact set containing $\gs(A)$ in its interior.
By \eqref{l0b}, we may choose $i_0$ such that if $i_0\le i\le j$, then
$|f_{i,j}(z)|\ge \frac12 |(j/i)^z|$ on $K$. If furthermore $i\le j\le 2i$, this
implies  $|f_{i,j}(z)|\ge c$ on $K$, for some $c>0$, and thus
$|f_{i,j}\qw(z)|\le c\qw$ on $K$. The result follows by \eqref{nC}.
(The condition $j\le 2i$ is not needed when $\gs(A)\subset\set{\Re z>0}$
so we may assume $\Re z\ge0$ for $z\in K$.)
\end{proof}

\section{Completions of the proofs}\label{Spf2}

\begin{proof}[Proof of \refT{TE}]
  By \eqref{EX} and \eqref{pl},
  \begin{equation}\label{ele}
	\E X_n = \sumgla F_{0,n} P_\gl X_0.
  \end{equation}
For each eigenvalue $\gl\neq\gl_1$, \refL{L1} shows that
\begin{equation}\label{gemini}
  F_{0,n} P_\gl X_0=
O\bigpar{n^{\Re\gl}\log^{\nu_\gl}n}
= O\bigpar{n^{\Re\gl_2}\log^{\nu_2}n}.
\end{equation}
Furthermore, by \eqref{p1},
\begin{equation}\label{matt}
  P_{\gl_1}X_0=(a\cdot X_0)v_1 = w_0v_1,
\end{equation}
and it follows from \eqref{l1bx} that
\begin{equation}\label{win}
  \begin{split}
  F_{0,n}P_{\gl_1}X_0 = \frac{n+w_0}{w_0}P_{\gl_1}X_0	
= \frac{n+w_0}{w_0} w_0 v_1
=(n+w_0) v_1.
  \end{split}
\end{equation}
The result \eqref{te} follows (when $\gl_1=1$)
 from \eqref{ele}, \eqref{gemini} and
\eqref{win}. 
\end{proof}

\begin{lemma}\label{LY0}
For every $n$,
$  P_{\gl_1}Y_n=0$.
\end{lemma}
\begin{proof}
Since the urn is balanced, $a\cdot \gD X_n=b$ is nonrandom, and thus, by  
\eqref{Yn}, 
\begin{equation}
a\cdot Y_n:=a\cdot\gD X_{n-1} - \E\bigpar{a\cdot\gD X_{n-1}\mid \cF_{n-1}}
=b-b=0.
\end{equation}
The result follows by \eqref{p1}.
\end{proof}

Using \eqref{pl}, we can rewrite \eqref{varx} as
\begin{equation}\label{varxx}
 \begin{split}
\Var(X_n)
=\sumgl\summu \sum_{i=1}^{n} F_{i,n}P_\gl \E\bigpar{Y_i Y_i'}P_\mu' F_{i,n}'.
  \end{split}
\end{equation}
For convenience, we define
\begin{equation}\label{tin}
\tin:=F_{i,n}P_\gl \E\bigpar{Y_i Y_i'}P_\mu' F_{i,n}'.
\end{equation}
Note that \refL{LY0} implies $P_{\gl_1}\E(Y_iY_i')=\E(P_{\gl_1}Y_iY_i')=0$
and thus also, by taking the transpose, $\E(Y_iY_i')P_{\gl_1}'=0$.
Hence $\tin=0$ when  $\gl=\gl_1$ or $\mu=\gl_1$, so these terms can be dropped
and \eqref{varxx} can be written
\begin{equation}\label{varxxx}
\Var(X_n)
=\sumglx\summux \sum_{i=1}^{n} \tin.
\end{equation}
We begin with a simple estimate of this sum.
The same estimates are given in \cite[Theorem 2.2]{BaiHu2005} under similar
conditions.
\begin{lemma}\label{LV}
  If $\gl_1=1$, then, for $n\ge2$,
  \begin{equation}\label{lv}
	\Var X_n
=
\begin{cases}
  O\bigpar{n},
& \Re\gl_2<\frac12,\\
  O\bigpar{n\log^{2\nu_2+1}n },
& \Re\gl_2=\frac12,\\
  O\bigpar{n^{2\Re\gl_2}\log^{2\nu_2}n },
& \Re\gl_2>\frac12.
\end{cases}
  \end{equation}
In particular,
  if $\gl_2<\gl_1=1$, then
  \begin{equation}\label{lva}
\Var(X_n) = o\bigpar{n^2}.	
  \end{equation}
\end{lemma}

\begin{proof}
It follows from \eqref{Exi2} that $\E (Y_nY_n')=O(1)$.
By combining this and \refL{L2}, we see that
if $\gl$ and $\mu$ are two eigenvalues, then,
for $1\le i\le n$,
\begin{equation}\label{jul1}
\tin=
 F_{i,n}P_\gl \E\bigpar{Y_i Y_i'}(F_{i,n}P_\mu)' 
=O\bigpar{(n/i)^{\Re\gl+\Re\mu}(1+\log (n/i))^{\nu_\gl+\nu_\mu}}.
\end{equation}
If $\Re\gl+\Re\mu\ge1$, we note that this implies
\begin{equation}\label{jul2}
\tin 
=O\bigpar{(n/i)^{\Re\gl+\Re\mu}\log^{\nu_\gl+\nu_\mu}n}
\end{equation}
while if $\Re\gl+\Re\mu<1$, we choose $\ga$ with $\Re\gl+\Re\mu<\ga<1$ and
note that \eqref{jul1} implies 
\begin{equation}\label{jul3}
\tin 
=O\bigpar{(n/i)^{\ga}}.
\end{equation}
By summing over $i$ we obtain from \eqref{jul2} and \eqref{jul3},
\begin{equation}\label{jul4}
  \begin{split}
\sumin \tin 
=
\begin{cases}
  O\bigpar{n},
& \Re\gl+\Re\mu<1,\\
  O\bigpar{n\log^{\nu_\gl+\nu_\mu+1}n },
& \Re\gl+\Re\mu=1,\\
  O\bigpar{n^{\Re\gl+\Re\mu}\log^{\nu_\gl+\nu_\mu}n },
& \Re\gl+\Re\mu>1.
\end{cases}
  \end{split}
\end{equation}
The result \eqref{lv}
follows from \eqref{varxxx} by summing \eqref{jul4} over the finitely many
$\gl,\mu\in\gsax$ and noting that our estimates are largest for
$\gl=\mu=\gl_2$. The simpler estimate \eqref{lva} is an immediate consequence.
\end{proof}

\begin{lemma}\label{LYlim}
If\/ $\Re\gl_2<\gl_1=1$,  then, as \ntoo,
\begin{equation}\label{lylim}
\E\bigpar{Y_nY_n'}\to B-v_1v_1'.  
\end{equation}
Hence, for any eigenvalue $\gl\neq\gl_1$, 
\begin{equation}\label{plylim}
P_\gl\E\bigpar{Y_nY_n'}\to P_\gl B.
\end{equation}
\end{lemma}

\begin{proof}
  By \eqref{Yn} and \eqref{mia},
$Y_{n+1}=\gD X_n-w_n\qw AX_n$, with $\E (Y_{n+1}\mid\cF_n)=0$ by \eqref{eyn}.
Hence,
\begin{equation}\label{ja}
  \begin{split}
	\E\bigpar{Y_{n+1}Y_{n+1}'\mid \cF_n}
=
	\E\bigpar{\gD X_n(\gD X_n)'\mid \cF_n} -w_n\qww A X_n(AX_n)'
  \end{split}
\end{equation}
and thus
\begin{equation}\label{jb}
  \begin{split}
	\E\bigpar{Y_{n+1}Y_{n+1}'}
=
	\E\bigpar{\gD X_n(\gD X_n)'} -w_n\qww A \E\bigpar{X_nX_n'}A'.
  \end{split}
\end{equation}

By the definition of the urn and \eqref{pin},
\begin{equation*}
  \begin{split}
\E\bigpar{\gD X_n(\gD X_n)'\mid \cF_n}
&=\sumjq \P\bigpar{I_{n+1}=j\mid \cF_n}\E\bigpar{\xi_j\xi_j'}
=\sumjq \frac{a_j X_{nj}}{w_n}\E\bigpar{\xi_j\xi_j'}
  \end{split}
\end{equation*}
and thus, using \eqref{wn2} and \refT{TE}, and recalling \eqref{B},
as \ntoo,
\begin{equation}\label{pib}
  \begin{split}
\E\bigpar{\gD X_n(\gD X_n)'}
=\sumjq \frac{a_j \E X_{nj}}{n+w_0}\E\bigpar{\xi_j\xi_j'}
\to
\sumjq a_j v_{1j}\E\bigpar{\xi_j\xi_j'}
= B.
  \end{split}
\end{equation}
Furthermore, by \eqref{lva} and \refT{TE} again,
\begin{equation}\label{pic}
n^{-2} \E\bigpar{X_nX_n'}
=n\qww \Var(X_n)+n\qww(\E X_n)(\E X_n)'
\to 0+v_1v_1'.
\end{equation}

Consequently, by \eqref{jb}, \eqref{pib}, \eqref{pic},
and recalling that $w_n/n\to1$ by
\eqref{wn2} and $Av_1=\gl_1v_1=v_1$,
\begin{equation}
\E\bigpar{Y_{n+1}Y_{n+1}'}
\to B-Av_1v_1'A'
= B-v_1v_1'.  
\end{equation}
This proves \eqref{lylim}, and \eqref{plylim} follows by noting that 
$P_\gl v_1=P_\gl P_{\gl_1}v_1=0$ when $\gl\neq\gl_1$.
\end{proof}

\begin{proof}[Proof of \refT{TV1}]
Let $\gl,\mu\in\gsax$, and note that, by our assumption,
$\Re\gl,\Re\mu\le\Re\gl_2<\frac12\gl_1=\frac12$. 
Write the inner sum in \eqref{varxxx} as an integral:
\begin{equation}\label{inte}
\frac{1}n  \sumin\tin =
\intoi \txn \dd x.
\end{equation}
For each fixed $x\in\ooi$, by Lemmas \ref{L3} and \ref{LYlim},
\begin{equation}\label{tlim}
  \begin{split}
\txn&=  F\xnn P_\gl \E\bigpar{Y_{\ceil{xn}}Y'_{\ceil{xn}}}P_\mu'F\xnn' 
\\&
\to 
x^{-A} P_{\gl} B P_\mu' x^{-A'}.	
  \end{split}
\end{equation}
Furthermore,
choose some $\ga\in[0,1)$ such that $\Re\gl_2<\frac12\ga$. 
Then, \eqref{jul3} applies and yields, for some $C<\infty$,
\begin{equation}
\begin{split}
  \txn 
\le C (n/\ceil{xn})^{\ga} 
\le C x^{-\ga},
\end{split}  
\end{equation}
which is integrable on $\ooi$. Thus, Lebesgue's theorem on dominated
convergence applies to \eqref{inte} and yields, by \eqref{tlim}
and the change of variables $x=e^{-s}$,
\begin{equation*}
\frac{1}n  \sumin\tin 
\to
\intoi  x^{-A} P_{\gl} B P_\mu' x^{-A'} \dd x
=\intoo  e^{sA} P_{\gl} B P_\mu' e^{sA'} e^{-s}\dd s.
\end{equation*}
Hence, \eqref{varxxx} and the definition \eqref{gSI} yield
\begin{equation*}
\frac{1}n\Var X_n=
\frac{1}n  \sumglx\summux\sumin\tin 
\to
\intoo  e^{sA} \PIx B \PIx' e^{sA'} e^{-s}\dd s
=\gS_I,
\end{equation*}
showing \eqref{tv1}.
\end{proof}

\begin{proof}[Proof of \refT{TV2}]
As in the proof of \refT{TV1},
we use \eqref{varxxx} and
consider the sum $\sumin\tin$ for two eigenvalues $\gl,\mu\in\gsax$.
By assumption, $\Re\gl+\Re\mu\le2\Re\gl_2=1$, and if $\Re\gl+\Re\mu<1$, then 
$\sumin\tin=O(n)$ by \eqref{jul4}. Hence we only have to consider the case
$\Re\gl+\Re\mu=1$, \ie, $\Re\gl=\Re\mu=\frac12=\Re\gl_2$.
In particular, $\nu_\gl,\nu_\mu\le\nu_2$.

In this case, as in \eqref{inte}, we
transform the sum into an integral, but this time in a somewhat
different way.
We have, using the change of variables $x=n^y=e^{y\log n}$,
\begin{equation}
  \begin{split}
\sumin\tin &= 
\tqn1 + \int_1^n \tqn{\ceil x}\dd x
\\&
=\tqn1 + \intoi \tyn  n^y \log n\dd y.	
  \end{split}
\end{equation}
Hence, since $\tqn1=O\bigpar{n\log^{2\nu_2}n}$ by \eqref{jul2},
\begin{equation}\label{benedictus}
  \begin{split}
\bigpar{n\log^{2\nu_2+1}n}\qw\sumin\tin 
& 
=o(1) + \intoi n^{y-1} (\log n)^{-2\nu_2}\tyn\dd y.	
  \end{split}
\end{equation}
Fix $y\in(0,1)$. Then, by \eqref{tedeum},
\begin{equation}\label{regina}
  \begin{split}
F\nyn P_\gl &= \frac{1}{\nu_2!}\parfrac{n}{\ceil{n^y}}^\gl  
\log^{\nu_2}\parfrac{n}{\ceil{n^y}} \Bigpar{N_\gl^{\nu_2}P_\gl	+o(1)}
\\&
= \frac{1}{\nu_2!}n^{(1-y)\gl}\bigpar{(1-y)\log n}^{\nu_2} 
\bigpar{N_\gl^{\nu_2}P_\gl	+o(1)}
  \end{split}
\end{equation}
and similarly for $\mu$. 

Recall the assumption $\Re\gl+\Re\mu=1$, and let $\tau:=\Im\gl+\Im\mu$, so
$\gl+\mu=1+\ii\tau$.
Then, by  \eqref{tin}, \eqref{regina} and \eqref{plylim},
  \begin{multline}\label{coeli}
n^{y-1} (\log n)^{-2\nu_2}\tyn
\\
= \frac{1}{(\nu_2!)^2}n^{\ii(1-y)\tau}(1-y)^{2\nu_2} 
N_\gl^{\nu_2}P_\gl B \bigpar{N_\mu^{\nu_2}	P_\mu}'+o(1).
\end{multline}
Moreover, by \eqref{jul2}, uniformly for $y\in\ooi$ and $n\ge2$,
\begin{equation}
  \begin{split}
 n^{y-1} (\log n)^{-2\nu_2}\tyn
= O\bigpar{(n/\ceil{n^y}) n^{y-1}} = O(1).	
  \end{split}
\end{equation}
Hence the error term $o(1)$ in \eqref{coeli} is also uniformly bounded, and
we can apply dominated convergence to the integral of it, yielding 
  \begin{multline}\label{rex}
\intoi  n^{y-1} (\log n)^{-2\nu_2}\tyn\dd y
\\
= \frac{1}{(\nu_2!)^2}\intoi n^{\ii(1-y)\tau}(1-y)^{2\nu_2} \dd y\cdot
N_\gl^{\nu_2}P_\gl {B} \bigpar{N_\mu^{\nu_2}	P_\mu}'+o(1).
\end{multline}
In the case $\tau=0$, \ie, $\mu=\bgl$, the integral on the \rhs{} of
\eqref{rex} is
$\intoi (1-y)^{2\nu_2}\dd y=(2\nu_2+1)\qw$. Furthermore, in this case,
$P_\mu=P_{\bgl}=\overline{P_\gl}$ and thus $P_\mu'=P_\gl^*$,
and similarly $N_\mu'=N_\gl^*$. Hence, \eqref{rex} yields
\begin{equation}\label{kyrie}
\intoi n^{y-1} (\log n)^{-2\nu_2}\tynb\dd y
= \frac{1}{(2\nu_2+1)(\nu_2!)^2}
N_\gl^{\nu_2}P_\gl {B} P_\gl^*\xpar{N_\gl^*}^{\nu_2}
+o(1).
\end{equation}
On the other hand, if $\tau\neq0$, then, with $u=1-y$,
\begin{equation}
  \intoi n^{\ii(1-y)\tau}(1-y)^{2\nu_2} \dd y
=
  \intoi e^{\ii (\tau\log n)u} u^{2\nu_2} \dd u
\to0
\end{equation}
as \ntoo{} and thus $|\tau\log n|\to\infty$, 
by an integration by parts
(or by the Riemann--Lebesgue lemma).
Hence, when $\mu\neq\bgl$,
\eqref{rex} yields
\begin{equation}\label{miseri}
  \intoi  n^{y-1} (\log n)^{-2\nu_2}\tyn\dd y
=o(1).
\end{equation}

We saw in the beginning of the proof that we can ignore 
the terms in \eqref{varxxx} with $\Re\gl<\frac12$ or $\Re\mu<\frac12$, and
by \eqref{benedictus} and \eqref{miseri}, we can also ignore the case 
$\Re\gl=\Re\mu=\frac12$ but $\mu\neq\bgl$. Hence only the case $\mu=\bgl$
with $\Re\gl=\frac12$ remains in \eqref{varxxx}, and the result follows by 
\eqref{benedictus} and \eqref{kyrie}.
\end{proof}

\begin{proof}[Proof of \refT{TL2}]
By \eqref{lva},
\begin{equation*}
\E\norm{X_n/n-\E X_n/n}^2=n^{-2}\E\norm{X_n-\E X_n}^2
=\sumiq n\qww\Var(X_{ni})\to0,
\end{equation*}
and
  $\E X_n/n\to v_1$ by \refT{TE}.
Hence,
$\E\norm{X_n/n-v_1}^2\to0$, which is the claimed convergence in $L^2$.

Moreover, if we fix 
$\eps\in(0,\frac12)$
such that $\Re\gl_2<1-\eps$, 
then the same argument shows, using \eqref{lv}
and \eqref{ele}--\eqref{win}, that, more precisely,
\begin{equation}\label{qin}
  \E\norm{X_n-(n+w_0)v_1}^2 = O\bigpar{n^{2-2\eps}}.
\end{equation}
Let $N\ge1$. 
By \eqref{kia} and the definition \eqref{qia},
for any $n\le N$, 
\begin{equation}\label{qid}
  F_{n,N}X_n = F_{0,N} X_0+\sum_{\ell=1}^{n}F_{\ell,N} Y_{\ell}.
\end{equation}
Moreover, 
by \eqref{eyn}, $Y_n$ is a martingale difference sequence, and
thus so is, for $n\le N$, $F_{n,N}Y_n$. Hence, \eqref{qid} shows that
$F_{n,N}X_n$, $n\le N$, is a martingale, and thus
\begin{equation}\label{ophelia}
  F_{n,N}X_n=\E\bigpar{X_N\mid\cF_n},\qquad n\le N.
\end{equation}
By \refL{L1b}, $F_{n,N} v_1=F_{n,N} \pli v_1=\frac{N+w_0}{n+w_0}v_1$ and thus
\eqref{ophelia} implies
\begin{equation}
 F_{n,N}\bigpar{X_n-(n+w_0)v_1}=\E\bigpar{X_N-(N+w_0)v_1\mid\cF_n},
\qquad n\le N.
\end{equation}
Hence, by Doob's inequality (applied to each coordinate) and \eqref{qin},
\begin{equation}
\E\sup_{n\le N}\norm{ F_{n,N}\bigpar{X_n-(n+w_0)v_1}}^2
\le 4\E\norm{X_N-(N+w_0)v_1}^2
=O\bigpar{N^{2-2\eps}}
\end{equation}
It follows, using \refL{LB}, that if $N\ge 2i_0$, then
\begin{equation}\label{poker}
\E\sup_{N/2\le n\le N}\norm{\bigpar{X_n-(n+w_0)v_1}/n}^2
=O\bigpar{N^{-2\eps}}.
\end{equation}
This holds trivially for smaller $N$ as well, since each $X_n\in L^2$ and
thus the \lhs{} of \eqref{poker} is finite for each $N$. 
Consequently, taking $N=2^k$ and
summing, 
\begin{equation}
  \E\sumk \sup_{2^{k-1}\le n\le 2^k}\norm{\bigpar{X_n-(n+w_0)v_1}/n}^2 <\infty.
\end{equation}
Consequently, $\norm{\bigpar{X_n-(n+w_0)v_1}/n}\to0$ a.s., and thus
$X_n/n\asto v_1$.
\end{proof}

\begin{proof}[Proof of \refT{TS}]
If $\Var(u\cdot X_n)=0$ for every $n$, then $u'\gS u=0$ by \eqref{agnus}.

For the converse, assume that $u'\gS u=0$. Then, by \eqref{tv1} and
\eqref{gSI},
\begin{equation}
0=u'\gS_I u=\intoo u'\PIx e^{sA} B e^{sA'} \PIx'u\, e^{-\gl_1 s} \dd s.
\end{equation}
The integrand is a continuous function of $s\ge0$, and non-negative since $B$
is non-negative definite by \eqref{B}. Hence, the integrand vanishes for
every $s\ge0$. In particular, taking $s=0$ we obtain, using \eqref{B} again,
\begin{equation}\label{sev}
  \begin{split}
0 &= u'\PIx B \PIx' u
=
  \sumiq a_i v_{1i} u'\PIx \E(\xi_i\xi_i')\PIx' u 
\\&
=
  \sumiq a_i v_{1i} \E\bigpar{u'\PIx \xi_i(u'\PIx\xi_i)'}
=
  \sumiq a_i v_{1i} \E\bigpar{u'\PIx \xi_i}^2,
  \end{split}
\end{equation}
noting that $u'\PIx \xi_i$ is a scalar.
Each term is non-negative, and thus each term is 0. 
If  $i$ is such that $a_i>0$, then it follows from the assumption that $\tA$
is irreducible that $v_{1i}>0$, and
hence \eqref{sev} yields
$\E\xpar{u'\PIx \xi_i}^2=0$ and thus $u'\PIx\xi_i=0$ a.s.
Furthermore, since the urn is balanced,
by \eqref{p1},
\begin{equation}
  P_{\gl_1}\xi_i = (a\cdot\xi_i)v_1 = b v_1.
\end{equation}
Hence, for every $i$ with $a_i>0$,
\begin{equation}
  \begin{split}
	u\cdot\xi_i = u\cdot \bigpar{\PIx+P_{\gl_1}}\xi_i
=0+u\cdot (bv_1)=bu\cdot v_1.
  \end{split}
\end{equation}
This is independent of $i$, and thus, for every $n$, a.s.,
\begin{equation}
  u\cdot \gD X_n=bu\cdot v_1.
\end{equation}
Consequently, a.s.,
\begin{equation}
u\cdot  X_n=u\cdot X_0+nbu\cdot v_1
\end{equation}
and thus $u\cdot X_n$ is deterministic.
\end{proof}

\section{Examples}\label{Sex}

\Polya{} urns have been used for a long time in various applications, for
example to study fringe structures in various random trees, see for example
\cite{BagchiPal}, \cite{AldousFP}, \cite{Devroye-local}.
Some recent examples are given in \cite{SJ292},
where, in particular, 
the number of two-protected nodes in a random $m$-ary search tree is studied for
$m=2$ and $3$ using suitable 
\Polya{} urns with 5 and 19 types, respectively, and it is shown
that
if this number is denoted by $Y_n$ ($m=2$) or $Z_n$ ($m=3$) for a search
tree with $n$ keys, then
\begin{align}\label{prot2}
\dfrac{Y_n-\frac{11}{30}n}{\sqrt{n}}
&\dto
N\lrpar{0,\frac{29}{225}},
\\
\frac{Z_n-\frac{57}{700}n}{\sqrt{n}}
&\dto
N\lrpar{0,\frac{1692302314867}{43692253605000}}.  \label{prot3}
\end{align}
(The binary case \eqref{prot2} had earlier been shown by \citet{MahmoudWard}
using other methods.)
The urns are strictly small; in both cases $\gl_1=1$ and $\gl_2=0$, with
$\nu_2=0$, and Theorems \ref{TE} and \ref{TV1} yield,
using the calculations in \cite{SJ292},
see \refR{RTV}, 
\begin{align}
\E Z_n&=\frac{57}{700}\,n+ O(1),
\\
\Var Z_n&
=
\frac{1692302314867}{43692253605000}\,n + o(n),
\end{align}
together with corresponding results for $Y_n$.
(The results for $Y_n$ were earlier shown in \citet{MahmoudWard},
where exact formulas for the mean and variance of $Y_n$ are given.)

Furthermore, \cite{SJ292} also studies the numbers of leaves 
and one-protected nodes
in a random 
$m$-ary search tree using a similar but simpler urn.
(For $m=2$ this was done already by \citet{Devroye-local}.)
For $2\le m\le 26$, this is a strictly small urn, and again the results in
\refS{Sresults} 
yield asymptotics of mean and variance. 

See \cite{SJ309} for further similar examples.

\begin{remark}\label{Rgaps-types}
  As said above, the urn used to show \eqref{prot2} has 5 types,
  corresponding to 5 different small trees.
To draw a ball corresponds to adding a node to a (randomly chosen) gap in
the corresponding tree; this may cause the tree to break up into several
smaller trees.
The 5 types have 4,3,2,1,0 gaps each, and these numbers are their
activities. Moreover, for type 2, the gaps are not equivalent, which makes
the replacement for this type random. 
(We have $\xi_2=(1,-1,0,0,0)$ with probability $1/3$
and $\xi_2=(0,0,0,1,0)$ with probability $2/3$, see \cite{SJ292}.)

A different, essentially equivalent, approach is to instead as types
consider the different gaps in the different trees;
this yields 5 new types that we denote by 1, 2A, 2B, 3, 4.
The transition from the old urn to the new is a simple linear
transformation: each old ball of type 1 is replaced by 4 new of type 1, 
which we write as $1\to4\cdot 1$, and
similarly $ 2\to \mathrm{2A}+2\cdot\mathrm{2B}$,
$3\to2\cdot 3$, $4\to4$, while balls of type 5 (which has activity 0) are
ignored.
This yields a new \Polya{} urn, where all types have activity 1. In the new
urn, all replacements are deterministic, which sometimes is an advantage, 
but on the other hand, replacements now may involve subtractions. 
For example, in the original urn, $\xi_1=(-1,1,1,0,0)$, meaning that if we
draw a ball of type 1, it is discarded and replaced by one of type 2 and one
of type 3. In the new urn, this translates to 
$\xi_1=(-4,1,2,2,0)$, meaning that we remove the drawn ball together with 3
others of type 1, and then add $\mathrm{2A}+2\cdot\mathrm{2B}+2\cdot3$.
Even worse, $\xi_2=(4,-1,-2,0,0)$, meaning that if we draw a ball of type
2A, we remove it together with two balls of type 2B, and add 4 balls of type
1. Nevertheless, by the construction, the urn is obviously tenable in the
sense of the present paper. This urn, with the gaps as types, thus is an
example of a tenable urn with subtractions that occur naturally in an
application.

The \Polya{} urn for the ternary search tree with 19 types in \cite{SJ292} 
can similarly be translated into an urn (with 29 types)
using gaps as types, again with
deterministic replacements, but sometimes subtractions.

See also \cite{SJ292}, where the transition to the corresponding urn with
gaps was used for the simpler urn used to study leaves; 
in that case there are no subtractions.
\end{remark}

\section{Further comments}\label{Sfurther}

The decomposition \eqref{kia} and its consequence \eqref{varx} 
explain some of the differences
between the small and large urns stated in the introduction.
Suppose again for convenience that $\gl_1=1$. Then, the term
$F_{\ell,n}Y_\ell$ in \eqref{kia}, which is the (direct and indirect) 
contribution from the $\ell$-th draw, has a variance
roughly (ignoring logarithmic factors when $\nu_2>0$) of the order 
$(n/\ell)^{2\Re\gl_2}$, see \refL{LV} and its proof.
For a large urn, this decreases rapidly with $\ell$
and $\sum_\ell \ell^{-2\Re\gl_2}$ converges, and thus the variance is
dominated by the contribution from the first draws. This strong long-term
dependency leads to the 
\as{} limit results, and the dependency of the limit on the initial
state $X_0$.

On the other hand, for a strictly small urn, the sum of the variances is of the
order $n$, but each term is $o(n)$ and is negligible, 
which explains why the first draws, and the initial state, do not affect
the limit distribution. 
In fact, 
for a component $X_{n,i}$ with asymptotic variance $(\gS)_{ii}>0$,
we see that for any $\eps>0$, all but a fraction
$\eps$ of $\Var X_{n,i}$ 
is explained by the draws with numbers in $[\gd n, n]$,
for some $\gd=\gd(\eps)>0$. The long-term dependency is thus weak in
this case.

The remaining case, a small urn with $\Re\gl_2=1/2$, is similar to the
strictly small case, but the long-term dependency is somewhat stronger.
If we for simplicity assume $\nu_2=0$, then the contribution of the $\ell$-th
draw to $\Var(X_n)$ is of the order $n/\ell$, giving a total variance of
order $n\log n$. 
Again, the first draws and the initial state do not affect
the limit distribution, but in order to explain all but a fraction $\eps$ of
the variance, we have to use the draws in 
$[n^\gd, n]$, for some small $\gd>0$.
(Cf.\ the functional limit theorem \cite[Theorem 3.31]{SJ154} with
different time scales for strictly small and non-strictly small urns.)

Cf.\ also \cite[Remark 4.3]{SJ154}, where a similar argument is made using the
corresponding continuous time branching process.

\appendix

\section{The largest eigenvalue}\label{Aten}

We have seen in \eqref{bala} that for a balanced urn, $b$ is an eigenvalue
of $A$, with a non-negative left eigenvector $a$. 

In typical applications, $b$ is the largest eigenvalue $\gl_1$.
Before proceeding, let us note that this is not always true.

\begin{example}\label{Ebad}
 As a counterexample, consider an urn with three colours, with activities
  1, and the  (deterministic) replacements 
$\xi_1=(1,2,0)$,
$\xi_2=(2,1,0)$,
$\xi_3=(-1,0,4)$.
The urn is balanced, with $b=3$, and if we start with $X_0=(1,0,0)$, the
urn is tenable. Nevertheless, the largest eigenvalue $\gl_1=4>b$.
Of course, the reason is that the urn never will contain any ball of colour
3, so this ought to be treated as an urn with just colours 1 and 2. (If
there is any ball of colour 3, the urn is not tenable.)
\end{example}

\refE{Ebad} is obviously a silly counterexample, but it shows that we need 
some extra assumption to exclude such trivialities.
We have the following result, which shows that if we only use colours that
actually can occur, then $\gl_1=b$ holds (or, at least, may be assumed) and
$\nu_1=0$. 

\begin{lemma}\label{Lapp}
  If the \Polya{} urn is tenable and balanced, 
and moreover any colour has a non-zero probability of ever appearing in the urn,
then $\Re\gl\le b$ for every
  $\gl\in\gs(A)$, and, furthermore, if\/ $\Re\gl=b$ then $\nul=0$.
We may thus assume $\gl_1=b$.
\end{lemma}

\begin{proof}
As in \refS{Smatrix}, we may and shall assume that $b=1$.

Suppose that $\gl\in\gs(A)$.
By \eqref{EX} and \refL{L1},
\begin{equation}\label{app}
  \begin{split}
\pl \E X_n 
&= \pl F_{0,n} X_0
= F_{0,n}\pl X_0
\\&
=n^{\gl}\log^{\nu_\gl}n \frac{\gG(w_0)}{\nul!\,\gG(w_0+\gl)}\nl^{\nul}\pl X_0
+ o\bigpar{n^{\Re\gl}\log^{\nu_\gl}n}.	
  \end{split}
\end{equation}

On the other hand, by our assumption \eqref{Exi2}, $\E\norm{\xi_i}<\infty$
for each $i$, and thus $\E\norm{\gD X_n}\le \max_i\E\norm{\xi_i}<\infty$
and therefore $\norm{\E X_n}\le\E\norm{X_n}=O(n)$.
Hence, 
\begin{equation}\label{app2}
\pl \E X_n=O(n).  
\end{equation}

Suppose now that either $\Re\gl>1$ or $\Re\gl=1$ and $\nul>0$.
Then \eqref{app} and \eqref{app2} yield the desired contradiction unless 
\begin{equation}\label{x00}
 \nl^{\nul}\pl X_0=0. 
\end{equation}
Moreover, if we run the urn for $k$ steps and regard
$X_k$ as a new starting position (conditioning on $X_k$), then the resulting
urn is \as{} tenable; hence the argument just given shows that
\begin{equation}
 \nl^{\nul}\pl X_k=0
\end{equation}
\as{} for every $k\ge0$.
Hence, also
$ \nl^{\nul}\pl \gD X_k=0$ a.s. 
If $j$ is any colour with $a_j>0$, then, by assumption, there exists a
$k$ such that $\P(X_{k,j}>0)>0$, and thus with positive probability $I_{k+1}=j$
and then $\gD X_k$ is a copy of $\xi_j$. Consequently, if $a_j>0$ then
\begin{equation}
 \nl^{\nul}\pl \xi_j=0
\end{equation}
\as, and thus
\begin{equation}
 \nl^{\nul}\pl \E\xi_j=0.
\end{equation}
In other words, for every $j$
\begin{equation}\label{xkj}
 \nl^{\nul}\pl \bigpar{a_j \E\xi_j}=0.
\end{equation}
However, by \eqref{A}, the $j$-th column of $A$ is
$a_j\E\xi_j$. Consequently,
\eqref{xkj} is equivalent to
$ \nl^{\nul}\pl A=0$.
Since $\pl A =\gl\pl+\nl$ by \eqref{24b}, and $\nl^{\nul+1}=0$,
this yields
\begin{equation}
0= \nl^{\nul}\pl A
=\gl \nl^{\nul}\pl +\nl^{\nul+1}
=\gl \nl^{\nul}
\end{equation}
and thus $\gl=0$, which contradicts the assumption on $\gl$.

Consequently, $\Re\gl\le1= b$ for every $\gl\in\gs(A)$, and since
$b\in\gs(A)$, $\Re\gl_1=\max_{\gl\in\gs(A)}\Re\gl=b$.
We have not ruled out the possibility that there are other eigenvalues $\gl$
with
$\Re\gl=b$, see \refR{Rbad} below, but even if this would happen, we have
shown that they all have $\nu_\gl=0$, so we are free to choose $\gl_1=b$.
\end{proof}

\begin{remark}
  As noted in \refR{Rmulti}, the eigenvalue $\gl_1=b$ may be multiple.
\refL{Lapp} shows that $\nu_1:=\nu_{\gl_1}=0$ also in this case.
\end{remark}

\begin{remark}\label{Rbad}
  \refL{Lapp} is not completely satisfactory since it does not rule out the
  possibility that besides $b$, there is also some complex eigenvalue $\gl=b+it$
  with $t\neq0$. We do not believe that this is possible, but we do not know
  a proof for a general tenable urn under our assumptions.
\end{remark}

\section{A note on \eqref{gSI}}\label{SSnote}
In the balanced case,
by \eqref{p1} and \eqref{balanced},
a.s.,
\begin{equation}
  \pli\xi_i=(a\cdot\xi_i)v_1=bv_1=\gl_1v_1,
\end{equation}
and thus, by \eqref{B} and \eqref{A},
\begin{equation}
  \begin{split}
\pli B &
=\sumiq a_i v_{1i}\E\bigpar{\pli\xi_i\xi_i'}
=\sumiq a_i v_{1i}\gl_1v_1\E\bigpar{\xi_i'}
=\gl_1v_1\Bigpar{\sumiq a_i v_{1i}\E\xi_{ij}}_j'
\\&
=\gl_1v_1\Bigpar{\sumiq (A)_{ji} v_{1i}}'_j
=\gl_1v_1\bigpar{A v_{1}}'
=\gl_1^2v_1v_{1}'.
  \end{split}
\raisetag{1.5\baselineskip}
\end{equation}
Since $B$ is symmetric, also
$B\pli'=(\pli B)'=\gl_1^2v_1v_1'$, and thus
\begin{equation}
  \pli B \pli' = \pli B = B \pli' = \gl_1^2 v_1v_1'
\end{equation}
and, as a simple consequence, still in the balanced case,
\begin{equation}\label{PBP}
  \PIx B \PIx' = \PIx B = B \PIx' = B-\gl_1^2 v_1v_1'.
\end{equation}
Cf.\ \refL{LYlim} (where $\gl_1=1$).

Hence, in the balanced case, we can omit either $\PIx$ or $\PIx'$ in
\eqref{gSI}. 
(This was noted empirically by Axel Heimb\"urger and Cecilia Holmgren,
personal communication.)
However, by \eqref{PBP}, 
we cannot omit both, nor even move both outside the integral,
because $e^{sA}v_1=e^{\gl_1s}v_1$ and thus
\begin{equation}
\intoo e^{sA}v_1v_1'e^{sA'}e^{-\gl_1s}\dd s
=\intoo e^{\gl_1s}v_1v_1'\dd s,   
\end{equation}
which diverges.

\begin{ack}
  I thank the anynomous referees for insightful and helpful comments.
\end{ack}

\newcommand\AAP{\emph{Adv. Appl. Probab.} }
\newcommand\JAP{\emph{J. Appl. Probab.} }
\newcommand\JAMS{\emph{J. \AMS} }
\newcommand\MAMS{\emph{Memoirs \AMS} }
\newcommand\PAMS{\emph{Proc. \AMS} }
\newcommand\TAMS{\emph{Trans. \AMS} }
\newcommand\AnnMS{\emph{Ann. Math. Statist.} }
\newcommand\AnnPr{\emph{Ann. Probab.} }
\newcommand\CPC{\emph{Combin. Probab. Comput.} }
\newcommand\JMAA{\emph{J. Math. Anal. Appl.} }
\newcommand\RSA{\emph{Random Struct. Alg.} }
\newcommand\ZW{\emph{Z. Wahrsch. Verw. Gebiete} }
\newcommand\DMTCS{\jour{Discr. Math. Theor. Comput. Sci.} }

\newcommand\AMS{Amer. Math. Soc.}
\newcommand\Springer{Springer-Verlag}
\newcommand\Wiley{Wiley}

\newcommand\vol{\textbf}
\newcommand\jour{\emph}
\newcommand\book{\emph}
\newcommand\inbook{\emph}
\def\no#1#2,{\unskip#2, no. #1,} 
\newcommand\toappear{\unskip, to appear}

\newcommand\arxiv[1]{\texttt{arXiv:#1}}
\newcommand\arXiv{\arxiv}

\def\nobibitem#1\par{}

\end{document}